\documentclass[a4paper,reqno,final]{amsart} %
\usepackage[utf8]{inputenc}
\usepackage[T1]{fontenc}
\usepackage{mathtools}

\usepackage{graphicx,mathrsfs,amssymb}%
\usepackage[usenames]{color}
\usepackage[colorlinks=true,citecolor=blue]{hyperref}
\usepackage{enumerate}

\newtheorem{theorem}{Theorem}[section]
\newtheorem{lemma}[theorem]{Lemma}
\newtheorem{corollary}[theorem]{Corollary}
\newtheorem{proposition}[theorem]{Proposition}
\theoremstyle{definition}
\newtheorem{definition}[theorem]{Definition}
\newtheorem{remark}[theorem]{Remark}

\numberwithin{equation}{section}

\newcommand{\R}{{\mathbb R}}

\newcommand{\norm}[1]{\left\lVert#1\right\rVert}

\DeclareMathOperator{\Id}{Id}

\renewcommand{\epsilon}{\varepsilon}
\renewcommand{\tilde}[1]{\widetilde{#1}}

\allowdisplaybreaks
\begin{document}
\title[The non-linear thermoelastic plate equation]{Local strong solutions for the non-linear thermoelastic plate equation on rectangular domains in $L^p$-spaces}

\author{S.\ Fackler}
\address{S.\ Fackler, Universit\"at Ulm, Institut f\"ur Angewandte Analysis, Ulm} %
\email{stephan.fackler@uni-ulm.de}

\author{T.\ Nau}
\address{T.\ Nau, Universit\"at Ulm, Institut f\"ur Angewandte Analysis, Ulm} %
\email{tobias.nau@uni-ulm.de}

\date{}

\begin{abstract}
In this article we consider the non-linear thermoelastic plate equation 
in rectangular domains $\Omega$.
More precisely, $\Omega$ is considered to be given as the Cartesian product of whole or half spaces and a cube. First the linearized equation is treated as an abstract Cauchy problem %
in $L^p$-spaces. We take advantage of the structure of $\Omega$ and apply operator-valued Fourier multiplier results to infer an $\mathcal R$-bounded $\mathcal H^\infty$-calculus %
for $A$. %
With the help of maximal $L^p$-regularity existence and uniqueness of local real-analytic strong solutions together with analytic dependency on the data is shown. %
\end{abstract}

\maketitle
\section{Introduction}\label{Sec_1}%
Let $n_1,n_2,n_3 \in \mathbb N_0$ such that $n_1 + n_2 + n_3 > 0$. We define the rectangular domain
\[
	\Omega := \mathbb R^{n_1} \times (0,\pi)^{n_2} \times (0,\infty)^{n_3}
\]
covering strips, half-strips and rectangles.
For $T > 0$ we consider the 
initial boundary value problem for the thermoelastic plate equation with a positive material constant $a \in (0,\infty)$ given by
\begin{equation}\label{TEP}
\begin{array}{r@{\quad=\quad}l@{\quad}l}
u_{tt} + \Delta^2 u + \Delta \theta + a\Delta(\Delta u)^3 & 0 & \text{in } (0,T) \times \Omega, \\
\theta_{t} -\Delta \theta - \Delta u_t & 0 & \text{in } (0,T) \times \Omega, \\
\Delta u \hspace{2mm} = \hspace{2mm} u \hspace{2mm}= \hspace{2mm} \theta &  0  & \text{on }(0,T) \times \partial \Omega,\\
u|_{t = 0} \hspace{2mm} = \hspace{2mm}  u_0,\hspace{2mm} u_t|_{t = 0} \hspace{2mm} =  \hspace{2mm} u_1,\hspace{2mm} 
\theta|_{t = 0} & \theta_0 & \text{in } \Omega.
\end{array}
\end{equation}
In our main result we prove unique local solvability of \eqref{TEP} in $L^p$ for $1 < p < \infty$.
We start our investigation with its non-homogeneous linearization
\begin{equation}\label{TEP_lin}
\begin{array}{r@{\quad=\quad}l@{\quad}l}
u_{tt} + \Delta^2 u + \Delta \theta  & g & \text{in } (0,T) \times \Omega, \\
\theta_{t} -\Delta \theta - \Delta u_t & h & \text{in } (0,T) \times \Omega, \\
\Delta u \hspace{2mm} = \hspace{2mm} u \hspace{2mm}= \hspace{2mm} \theta &  0  & \text{on }(0,T) \times \partial \Omega,\\
 u|_{t = 0} \hspace{2mm} = \hspace{2mm}  u_0,\hspace{2mm} u_t|_{t = 0} \hspace{2mm} =  \hspace{2mm} u_1,\hspace{2mm} 
\theta|_{t = 0} & \theta_0 & \text{in } \Omega
\end{array}
\end{equation}
and prove maximal $L^p$-regularity of an associated abstract first order Cauchy problem.
In fact, at the very beginning the conditions in lines~3~and~4 of \eqref{TEP} and \eqref{TEP_lin} will be replaced by
\begin{equation}\label{TEP_lin_add_2}
\begin{array}{r@{\quad=\quad}l@{\quad}l}
\Delta u \hspace{2mm} = \hspace{2mm} u_t \hspace{2mm}= \hspace{2mm} \theta &  0  & \text{on }(0,T) \times \partial \Omega,\\
\Delta u|_{t = 0} \hspace{2mm} = \hspace{2mm} \tilde u_0, \hspace{2mm} u_t|_{t = 0} \hspace{2mm} =  \hspace{2mm} u_1,
	\hspace{2mm} \theta|_{t = 0} & \theta_0 & \text{in } \Omega.
\end{array}
\end{equation}
This admits a comfortable transformation of \eqref{TEP} to a non-linear first order system
\begin{equation}\label{TEP_first_order}
	 U_t + A(D)U = \Phi(U) \ \text{ in }\ (0,T) \times \Omega,
	 \quad U|_{(0,T) \times \partial \Omega} = 0,
	 \quad U|_{t = 0} = U_0 
\end{equation}
and of \eqref{TEP_lin} to the according linearized non-homogeneous linear first order system
\begin{equation}\label{TEP_lin_first_order}
	  U_t + A(D)U = F \ \text{ in }\ (0,T) \times \Omega,
	 \quad U|_{(0,T) \times \partial \Omega} = 0,
	 \quad U|_{t = 0} = U_0
\end{equation}
respectively. We refer to Section \ref{Sec_Main_Theorems} for the precise transformation and the definitions of $U$, $A(D)$, $\Phi(U)$, and $F$. 
The strategy we pursue is to investigate the linearized system \eqref{TEP_lin_first_order} first, however,
subject to periodic boundary conditions in the infinite rectangular cylinder $\mathbb R^{n_1} \times (0,2 \pi)^{n_2}$. This allows for an application of powerful $L^p$-Fourier multiplier theorems. Afterwards, with the aid of reflection techniques the linearized system \eqref{TEP_lin_first_order} in the space domain $\Omega$ is treated. To prove  maximal $L^p$-regularity of the associated operator we employ a celebrated result from \cite{Weis-2001} roughly saying that the property of $\mathcal R$-sectoriality is equivalent to maximal $L^p$-regularity.
With the help of maximal $L^p$-regularity the non-linear system obtained from equation \eqref{TEP} is treated using the Banach fixed point theorem.
Finally, we transfer our results on the first order systems (\eqref{TEP_first_order} - \eqref{TEP_lin_first_order}) to the original systems (\eqref{TEP} - \eqref{TEP_lin}). This can be done e.g.~in the case that $\Omega \subset \mathbb R^2$ does not define the whole or a half space, i.e.~ provided $n_2 \neq 0$.

The non-linear model \eqref{TEP} for the thermoelastic plate equation also appears in \cite{Lasiecka-Maad-Sasane-2008}. There the motivation and derivation from physics is recalled briefly. 
We further refer to the overview on recent results on thermoelasticity given there. The appendix in \cite{Lasiecka-Maad-Sasane-2008} comments on an approach to the non-linear thermoelastic plate equation based on maximal regularity in the context of H\"older spaces for the linearized system. Here results from \cite{Lunardi-1984} and \cite{Lunardi-1995} are employed. In contrast to the theory of H\"older spaces we present a maximal regularity approach via the notion of $\mathcal R$-sectoriality in $L^p$-spaces. Note that maximal $L^p$-regularity implies the generation of an analytic $C_0$-semigroup.  
A generalization of the linearized thermoelastic plate equation in $L^p(\mathbb R^n)$ was considered in \cite{Denk-Racke-2006}. There the Newton-polygon method is used to deal with the mixed order structure of the so-called $\alpha$-$\beta$-system. The authors prove analyticity of the associated $C_0$-semigroup and decay properties of the solution. Results for the half-space problem in $L^p$-spaces supplemented with generalized Dirichlet boundary conditions are proved in \cite{Naito-Shibata-2009}. Bounded and exterior domains with boundaries of class $C^4$ are added later in the articles \cite{Denk-Racke-Shibata-2009} and \cite{Denk-Racke-Shibata-2010}. For earlier results in $L^p$-spaces we refer the reader to the references given there. It is worthwhile to mention that our setting includes unbounded domains with non-compact boundary and corners at the same time. The case of a rectangular plate e.g. is not only new in the $L^p$-setting but also very relevant from the perspective of physical
applications.
Actually, besides the type of boundary conditions it is this special structure of the underlying space domain $\Omega$ which allows for the multiplier approach to be carried out in the article in hand. Similar results to ours on \eqref{TEP} and \eqref{TEP_lin} 
seem to be available adapting the Newton-polygon method from \cite{Denk-Racke-2006} and the reflection arguments from below.

Recently and independently of our work the thermoelastic plate equation is studied in~\cite{Lasiecka-Wilke-2013} also using maximal $L^p$-regularity theory. There the authors consider bounded domains with $C^2$-boundary, whereas our methods are adapted to the physical relevant case of rectangular domains. Moreover, for $C^2$-boundaries one can rely on maximal regularity results developed in~\cite{Denk-2003}, whereas for domains with corners new methods are necessary as developed below.

\section{Roadmap \& the main theorems}\label{Sec_Main_Theorems}
We start this section with the detailed description of the transformation of lines~1~and~2 in \eqref{TEP} and \eqref{TEP_lin} supplemented with the conditions \eqref{TEP_lin_add_2} to first order problems as introduced in \cite{Liu-Renard-1995}. Having set
\begin{equation} \label{Trans}
U := (U_1,U_2,U_3)^T := (\Delta u, u_t, \theta)^T
\quad \text{and} \quad
F := (0,g,h)^T
\end{equation}
we rewrite lines~1~and~2 of \eqref{TEP_lin} as
\begin{equation}\label{Def_M}
 U_t -
 M\Delta U
	= F,
	\quad \text{where} \quad
	M := 
	\begin{pmatrix}
	 0& 1& 0\\
	 -1& 0& -1\\
	 0& 1& 1
	\end{pmatrix}.
\end{equation}
Looking at \eqref{TEP_lin_add_2}, the initial conditions turn to $U|_{t = 0} = U_0$, where $U_0 := (\tilde u_0, u_1, \theta_0)^T$ and the boundary conditions imply Dirichlet boundary conditions for the transformed system \eqref{Def_M}, i.e.~$U|_{(0,T) \times \partial \Omega} = 0$. 
Following \cite{Lasiecka-Maad-Sasane-2008} we let $\phi(s) := s^3$. Then the transformed matrix equation associated with lines 1~and~2 of the non-linear thermoelastic plate equation \eqref{TEP} in $\Omega$ reads as
\begin{equation}\label{Def_Phi}
 U_t -
 M\Delta U = \Phi(U),
	\quad \text{where} \quad
	\Phi (U) := -a\Delta 
	\begin{pmatrix}
	 0\\
	 \phi(U_1)\\
	 0
	\end{pmatrix}.
\end{equation}
In view of Fourier transform and Fourier coefficients we introduce the differential expression $D_j := -i \partial_j$, where $i$ denotes the imaginary unit. Employing the notation $A(D) := M \sum_{j = 1}^nD_j^2$ we are led to problems \eqref{TEP_first_order} and \eqref{TEP_lin_first_order}.

We are now in the position to state the main theorems of this article. Let $1 < p < \infty$. We realize the matrix differential operator $A(D)$ in $L^p(\Omega)^3$ as
\begin{equation}\label{Real}
\begin{aligned}
 D(A) & :=  \big(W^{2,p}(\Omega) \cap  W^{1,p}_0(\Omega)\big)^3,\\
 A u & :=  A(D) u \quad (u \in D(A)).
\end{aligned}
\end{equation}
As a first result we prove that $A$ has an $\mathcal R$-bounded $\mathcal H^\infty$-calculus which implies $\mathcal R$-sectoriality of $A$ (see Definition \ref{Def_Rversions}). Moreover, we prove invertibility of $A$ in the case that $\Omega$ defines a strip, a half-strip or a rectangle. 
\begin{theorem}\label{Hinf}
Let $1 < p < \infty$. Then $A \in \mathcal {RH}^\infty(L^p(\Omega)^3)$ and $\phi_{A}^{\mathcal R \infty} < \frac{\pi}{2}$.
In particular, $A$ is injective. Moreover, if $n_2 \neq 0$ then $0 \in \rho(A)$.
\end{theorem}
By \cite[Theorem 4.2]{Weis-2001} we obtain the following implication on maximal $L^p$-regularity, $1 < p  <\infty$, for the Cauchy problem 
\begin{equation}\label{CP}
\begin{aligned}
	\dot U + AU & = F \quad \text{on } (0,T),\\
	U(0) & = U_0,
\end{aligned}
\end{equation}
obtained from the initial boundary value problem \eqref{TEP_lin} of the linearized thermo\-elastic plate equation.
In what follows 
\[
	I_p(A) := \{U_0 \in L^p(\Omega)^3: \exists U \in W^{1,p}((0,T),L^p(\Omega)^3) \cap L^p((0,T),D(A)): U_0 = U(0) \}
\]
supplemented with the norm
\[
	\|U_0\|_{I_p(A)} := \inf\{\|U\|_{W^{1,p}((0,T),L^p(\Omega)^3) \cap L^p((0,T),D(A))} : U_0 = U(0)\}
\]
denotes the space of all possible traces at time zero.
\begin{corollary}\label{Cor_MR}
In the situation of Theorem \ref{Hinf} the realization $A$ has the property of maximal $L^p$-regularity:
For every $T \in (0,\infty)$, every $F \in L^p((0,T),L^p(\Omega)^3)$ and every $U_0 \in I_p(A)$ %
problem \eqref{CP} associated with \eqref{TEP_lin_first_order} has a unique solution $U \in W^{1,p}((0,T),L^p(\Omega)^3) \cap L^p((0,T),D(A))$
and there exists a $C >0$ independent of $F$ and $U_0$ such that
\[
	\|U\|_{W^{1,p}((0,T),L^p(\Omega)^3) \cap L^p((0,T),D(A))}
	\leq C \bigl(\|F\|_{L^p((0,T),L^p(\Omega)^3)} + \norm{U_0}_{I_p(A)} \bigr).
\]
Moreover if $n_2 \neq 0$, %
the assertion remains valid for $T =\infty$.
\end{corollary}
With the help of maximal $L^p$-regularity we prove existence and uniqueness of strong solutions to the non-linear problem \eqref{TEP_first_order}.
For the remainder of this section let $n := n_1 + n_2 + n_3$.
\begin{theorem}\label{Non-lin}
Let $n + 2 < p < \infty$.

a)\ For given $U_0 \in I_p(A)$ there exists $T_0 \in (0,\infty)$ such that problem \eqref{TEP_first_order}  
has a unique solution $ U \in W^{1,p}((0,T),L^p(\Omega)^3) \cap L^p\big((0,T),(W^{2,p}(\Omega) \cap  W^{1,p}_0(\Omega))^3\big)$ for each $T \in (0,T_0)$.

b)\ Accordingly, for every $T \in (0,\infty)$ there exists $d > 0$ such that \eqref{TEP_first_order} 
has a unique solution $ U \in W^{1,p}((0,T),L^p(\Omega)^3) \cap L^p\big((0,T),(W^{2,p}(\Omega) \cap  W^{1,p}_0(\Omega))^3\big)$ provided $\|U_0\|_{I_p(A)} \leq d$.

In both cases the solution $U$ depends analytically on $U_0$.
\end{theorem}
Again with the help of maximal $L^p$-regularity we prove that the solution we have found in Theorem \ref{Non-lin} is real-analytic.
\begin{theorem}\label{Analytic}
Let $n + 2 < p < \infty$. Then the solution $U = U(t,x)$ established in Theorem~\ref{Non-lin} is real-analytic, i.e., $U \in C^{\omega}((0, T) \times \Omega,\mathbb R^3)$.
\end{theorem}
Finally, we transfer the results on \big(\eqref{TEP_first_order} - \eqref{TEP_lin_first_order}\big) to the systems \big(\eqref{TEP}- \eqref{TEP_lin}\big).
Here we assume $n_2 \neq 0$ which ensures $0 \in \rho(A)$ by Theorem~\ref{Hinf}. Thus, we may assume $T = \infty$ in Corollary \ref{Cor_MR}. 
Furthermore, it will subsequently allow us to invert the transformation carried out in \eqref{Trans}. In particular, it implies $0 \in \rho(\Delta^D_p)$, where $\Delta^D_p$ denotes the $L^p(\Omega)$-realization of the Dirichlet-Laplacian with domain $D(\Delta^D_p) := W^{2,p}(\Omega) \cap  W^{1,p}_0(\Omega)$ (cf.~ \cite{Nau-2012} or \cite{Nau-Diss-2012}). Note that $M \in \mathbb R^{3\times 3}$ is a regular matrix and that $D(A) = D(\Delta^D_p)^3$. As usual we define the square $[\Delta^D_p]^2$ of the closed operator $\Delta^D_p$ by
\[
	D([\Delta^D_p]^2) := \{ u \in D(\Delta^D_p) : \Delta^D_p u \in D(\Delta^D_p)\},
	\quad [\Delta^D_p]^2 u := \Delta^D_p ( \Delta^D_p u).
\]
Then
\[
	D([\Delta^D_p]^2) = \{ u \in W^{4,p}(\Omega) \cap  W^{1,p}_0(\Omega): \Delta u  = 0\}	
\]
and the induced operator
\[
	\Delta^D_{2,p} \colon D([\Delta^D_p]^2) \subset  D(\Delta^D_p) \to D(\Delta^D_p),\ \Delta^D_{2,p} u := \Delta^D_{p} u
\]
fulfills $0 \in \rho(\Delta^D_{2,p})$ and $(\Delta^D_{2,p})^{-1} \in \mathcal L(D(\Delta^D_p),D([\Delta^D_p]^2))$.

Moreover, $I_p(A)$ is characterized by real interpolation: we have
$I_p(A) = \big(L^p(\Omega)^3,D(A)\big)_{1-1/p,p} = \big(L^p(\Omega),D(\Delta^D_p)\big)_{1-1/p,p}^3$
(\cite[Corollary 1.14]{Lunardi-2009}).
\begin{remark}
For $\frac32 < p < \infty$ as considered below and $\Omega := (0,\pi) \times \mathbb R^{n_1}$ defining a strip more can be said on the trace space $I_p(A)$. In that case the space $B^{2-2/p}_{p,p,(0)}(\Omega) := \{u \in B^{2-2/p}_{p,p}(\Omega) : u|_{\partial \Omega} = 0\}$ is well-defined (cf.~\cite[Section~4.3.3]{Triebel-1978}) and results on half spaces and standard localization techniques can be employed to show $\big(L^p(\Omega),D(\Delta^D_p)\big)_{1-1/p,p} = B^{2-2/p}_{p,p,(0)}(\Omega)$. Here $B^{s}_{p,q}(\Omega)$ denotes the Besov space with parameters $s$, $p$, and $q$.
\end{remark}
To state our results on the systems \big(\eqref{TEP}- \eqref{TEP_lin}\big) we define the spaces
\[
	\mathcal B_p := \big(L^p(\Omega),D(\Delta^D_p)\big)_{1-1/p,p}, \quad \|w\|_{\mathcal B_p} := \|w\|_{(L^p(\Omega),D(\Delta^D_p))_{1-1/p,p}}
\]
and
\[
	\mathcal A_p := \big[\Delta_p^D\big]^{-1}\big(\mathcal B_p\big), \quad \|w\|_{\mathcal A_p} := \|\Delta w\|_{\mathcal B_p}.
\]
For convenience of the reader being merely interested in \eqref{TEP} we do not postpone the proof of our main theorem which reads as follows. 
\begin{theorem}\label{Re_Trans}
a)\ Let $1 < p < \infty$, let $n_2 \neq 0$, and let $T \in (0,\infty]$. Then for all $(g,h) \in L^p((0,T),L^p(\Omega))^2$ and all $(u_0,u_1,\theta_0) \in \mathcal A_p \times \mathcal B_p \times \mathcal B_p$ there exists a unique solution
\begin{eqnarray*}
	u & \in & W^{2,p}((0,T),L^p(\Omega)) \cap W^{1,p}((0,T),D(\Delta^D_p)) \cap L^p((0,T),D([\Delta^D_p]^2))\\
	\theta & \in & W^{1,p}((0,T),L^p(\Omega)) \cap L^p((0,T),D(\Delta^D_p))
\end{eqnarray*}
of the linearized problem \eqref{TEP_lin}. Moreover, there exists $C > 0$ independent of $g$, $h$, $u_0$, $u_1$, and $\theta_0$ such that
\begin{equation}\label{TEP_MR}
\begin{aligned}
\|u & \|_{W^{2,p}((0,T),L^p(\Omega))}   + \|u\|_{W^{1,p}((0,T),W^{2,p}(\Omega))} + \|u\|_{L^p((0,T),W^{4,p}(\Omega))}\\
	& \quad + \|\theta \|_{W^{1,p}((0,T),L^p(\Omega))}  + \|\theta\|_{L^p((0,T),W^{2,p}(\Omega))}\\
	& \quad \qquad \qquad \leq  C \big(\|(g,h)\|_{L^p((0,T),L^p(\Omega))^2} + \|(u_0,u_1,\theta_0)\|_{\mathcal A_p \times \mathcal B_p\times \mathcal B_p}\big).
\end{aligned}
\end{equation}
b)\ Assume both $n + 2 < p < \infty$ and $n_2 \neq 0$ and let $(u_0,u_1,\theta_0) \in \mathcal A_p \times \mathcal B_p \times \mathcal B_p$.

Then there exists $T_0 \in (0,\infty)$ such that problem \eqref{TEP} has a unique solution
\begin{eqnarray*}
	u & \in & W^{2,p}((0,T),L^p(\Omega)) \cap W^{1,p}((0,T),D(\Delta^D_p)) \cap L^p((0,T),D([\Delta^D_p]^2))\\
	\theta & \in & W^{1,p}((0,T),L^p(\Omega)) \cap L^p((0,T),D(\Delta^D_p))
\end{eqnarray*}
for each $T \in (0,T_0)$.

Accordingly, for each $T \in (0,\infty)$ there exists $d > 0$ such that problem \eqref{TEP} has a unique solution
\begin{eqnarray*}
	u & \in & W^{2,p}((0,T),L^p(\Omega)) \cap W^{1,p}((0,T),D(\Delta^D_p)) \cap L^p((0,T),D([\Delta^D_p]^2))\\
	\theta & \in & W^{1,p}((0,T),L^p(\Omega)) \cap L^p((0,T),D(\Delta^D_p))
\end{eqnarray*}
provided $\|(u_0,u_1,\theta_0)\|_{\mathcal A_p \times \mathcal B_p \times \mathcal B_p} < d$.

Moreover, in both cases $u, \theta \in C^{\omega}((0, T)\times \Omega,\mathbb R)$ are real-analytic and depend real-ana\-lyti\-cal\-ly on $u_0$, $u_1$, and $\theta_0$. 
\end{theorem}
\begin{proof}
a)\ Let $U_0 := (\Delta u_0,u_1,\theta_0)$ and $F := (0,g,h)$. By Corollary \ref{Cor_MR}
there exists a unique solution $ U \in W^{1,p}((0,T),L^p(\Omega)^3) \cap L^p\big((0,T),(W^{2,p}(\Omega) \cap  W^{1,p}_0(\Omega))^3\big)$ of \eqref{TEP_lin_first_order}. Let $U = (v_1,v_2,v_3)$. We evaluate the first line in \eqref{TEP_lin_first_order} to the result
\[
	\partial_t v_1 = \Delta v_2 \quad \text{in} \quad L^p((0,T),L^p(\Omega)).
\]
Since  $(\Delta^D_p)^{-1} \in \mathcal L ( L^p(\Omega) )$ commutes with the time derivative $\partial_t$ this implies $\partial_t \big((\Delta^D_p)^{-1}v_1\big) = v_2$.
With $v_1 \in W^{1,p}((0,T),L^p(\Omega)) \cap L^p((0,T),D(\Delta^D_p))$ we define $u := (\Delta^D_p)^{-1}v_1$ and obtain
$u \in W^{1,p}((0,T),D(\Delta^D_p)) \cap L^p((0,T),D([\Delta^D_p)]^2)$.
Thanks to $\partial_t v_1 = \Delta v_2$ we have
$\partial_{tt} u = \partial_{tt} (\Delta^D_p)^{-1}v_1 = \partial_t (\Delta^D_p)^{-1}\partial_t v_1 = \partial_t v_2$. Altogether we have found
\[
	u \in W^{2,p}((0,T),L^p(\Omega)) \cap W^{1,p}((0,T),D(\Delta^D_p)) \cap L^p((0,T),D([\Delta^D_p]^2))
\]
and $(\Delta^D_p)^{-1} \in \mathcal L\big(L^p(\Omega), W^{2,p}(\Omega)\big)$ and $(\Delta^D_{2,p})^{-1} \in \mathcal L(D(\Delta^D_p),D([\Delta^D_p]^2))$ yields \eqref{TEP_MR}.
Plugging in we deduce that $u$ and $\theta := v_3$ fulfill lines 1~and~2 of \eqref{TEP_lin} as well as the conditions prescribed in \eqref{TEP_lin_add_2}. Additionally, $u \in L^p((0,T),D([\Delta^D_p]^2))$ implies $\Delta u = u = 0$ in $L^p((0,T),L^p( \partial \Omega))$. Furthermore,  from the definition $u := (\Delta^D_p)^{-1}v_1$ we deduce
\[
	u|_{t = 0} = (\Delta^D_p)^{-1}v_1|_{t = 0} = (\Delta^D_p)^{-1} (v_1|_{t = 0}) = u_0
\]
and
\[
	\partial_t u|_{t = 0} = (\Delta^D_p)^{-1}\partial_t v_1|_{t = 0} = v_2|_{t = 0} = u_1.
\]
Recall that taking the trace as well as partial derivation with respect to $t$ commutes with $(\Delta^D_p)^{-1}$. 
Thus, $u$ and $\theta$ define a solution of \eqref{TEP_lin} and uniqueness follows from the uniqueness assertion on $U$ as given in Corollary \ref{Cor_MR}.\\
b)\ We replace $F$ by $\Phi(U)$. Along the same lines as in part a) of the proof above Theorem~\ref{Non-lin} yields unique solvability of \eqref{TEP}. The final assertion on analytic dependancy and analyticity of $u$ and $\theta$ follows from Theorem~\ref{Non-lin}, Theorem~\ref{Analytic}, and the definition $u := (\Delta^D_p)^{-1}v_1$.
\end{proof}

\section{Function spaces and classes of operators}\label{Sec_FSandCO}%

For $1 < p <\infty$ we denote the Lebesgue-Bochner spaces by $L^p(G,E)$ where $G$ is a domain and $E$ denotes a Banach space.
Let $m \in \mathbb N_0$. The $E$-valued Sobolev space $W^{m,p}(G,E)$ of order $m$ consists of all $u \in L^p(G,E)$ such that all distributional derivatives up to order $m$ define functions in $L^p(G,E)$. Since we will have to deal with periodic boundary conditions on rectangular domains $G = \mathcal Q_{n_2} \coloneqq (0,2\pi)^{n_2}$, periodic Sobolev spaces $W^{m,p}_{per}(\mathcal Q_{n_2},E)$ are of special interest. They consist of all $u \in W^{m,p}(\mathcal Q_{n_2},E)$ such that
\[
				\partial_j^\ell u|_{x_j = 0} =  \partial_j^\ell u|_{x_j = 2\pi}\quad (j = 1,\ldots,n_2;\ 0 \leq \ell < m).
\]
These traces are well-defined by continuity. Indeed, we have
\[
	W^{m,p}(\mathcal Q_{n_2},E) %
	 \hookrightarrow L^p(\mathcal Q_{n_2-1},C^{m-1}([0,2\pi],E)).
\]
By definition $W^{0,p}_{per}(\mathcal Q_{n_2},E) = L^p(\mathcal Q_{n_2},E)$. If $E = \mathbb C$ we agree to drop the additional indication in the definitions above and write as usual $L^p(G)$, for instance.
Some theorems we will apply need the UMD property or property $(\alpha)$ for which we refer to~\cite[Sections~1~and~4]{Kunstmann-Weis-2004} for details. For what follows it will be enough to know that $L^p$-spaces for $1 < p < \infty$ share both properties. %

For $0 < \gamma < 1$ and $p$ as above let $\big(L^p(G),W^{m,p}(G)\big)_{\gamma,p}$ denote the real interpolation space of $L^p(G)$ and  $W^{m,p}(G)$. Then one defines the Besov spaces with parameters $0 < s < \infty$, $1 \le p < \infty$ and $1 \le q \le \infty$ as $B^{s}_{p,q}(G) \coloneqq (L^p(G), W^{m,p}(G))_{s/m,q}$, where $m$ is the smallest integer larger than $s$, and let $W^s_p(G) := B^{s}_{p,p}(G)$. Then $\big(L^p(G),W^{m,p}(G)\big)_{\gamma,p} = W^{m\gamma}_{p}(G)$ and, in particular, $\big(L^p(G),W^{m,p}(G)\big)_{1-1/p,p} = W^{m-m/p}_{p}(G)$ (see e.g.~\cite[Section 4.3.1]{Triebel-1978}).

Given two Banach spaces $X$ and $Y$ we write $\mathcal{L}(X,Y)$ for the space of bounded, linear operators from $X$ to $Y$ and abbreviate $\mathcal{L}(X) := \mathcal{L}(X,X)$.
\begin{definition}\label{def_rbdness} 
Let $X$ and $Y$ be Banach spaces. A family $\mathcal{T} \subset \mathcal{L}(X,Y)$ is called 
{\em $\mathcal{R}$-bounded} if there exist a $C > 0$ and a $p \in [1,
\infty)$ such that for all $N \in \mathbb{N}, T_j \in \mathcal{T}, x_j
\in X$ and all independent symmetric $\{-1,1\}$-valued random variables
$\varepsilon_j$ on a probability space $(G,\mathcal{M},P)$ for $j = 1,...,N$ we have
\begin{equation}\label{DefRBeschr}
\| \sum\limits_{j=1}^{N} \varepsilon_j T_jx_j\|_{L^p(G,Y)} \leq C \|
\sum\limits_{j=1}^{N} \varepsilon_j x_j\|_{L^p(G,X)}.
\end{equation}
The smallest $C>0$ such that \eqref{DefRBeschr} is satisfied is 
called the \emph{$\mathcal R$-bound} of $\mathcal T$ and denoted by $\mathcal
R_p(\mathcal T)$.
\end{definition}
In contrast to the property of $\mathcal R$-boundedness itself, the $\mathcal R$-bound $\mathcal R_p(\mathcal T)$ is not independent of $p \in [1,\infty)$. For our purposes, however, $p$-dependancy is not important and we agree to write $\mathcal R(\mathcal T)$.
We will frequently use the fact that $\mathcal{R}$-bounds essentially behave like uniform norm bounds as formulated together with the contraction principle of Kahane in the next lemma. 
\begin{lemma}\label{HintSumme}
a) \ Let $X$, $Y$ and $Z$ be Banach spaces and let $\mathcal{T},
\mathcal{S} \subset \mathcal{L}(X,Y)$ and $\mathcal{U} \subset
\mathcal{L}(Y,Z)$ be $\mathcal{R}$-bounded. Then $\mathcal{T} +
\mathcal{S} \subset \mathcal{L}(X,Y)$, $\mathcal{T} \cup
\mathcal{S} \subset \mathcal{L}(X,Y)$, and $\mathcal{U} \mathcal{T}
\subset \mathcal{L}(X,Z)$ are $\mathcal{R}$-bounded as well and we have
$$\mathcal{R}(\mathcal{T} + \mathcal{S}), \ \mathcal{R}(\mathcal{T} \cup \mathcal{S}) \leq
\mathcal{R}(\mathcal{S})+\mathcal{R}(\mathcal{T}), \quad
\mathcal{R}(\mathcal{U}  \mathcal{T}) \leq
\mathcal{R}(\mathcal{U})\mathcal{R}(\mathcal{T}).$$
b) \
Let $p \in[1,\infty)$. Then for all $N\in\mathbb{N}, x_j\in X,
\varepsilon_j$ as in Definition~\ref{def_rbdness} 
and for all $a_j, b_j\in\mathbb{C}$ with
$|a_j|\leq|b_j|$ for $j=1,\dots,N$ we have
\begin{equation}
\|\sum_{j=1}^N a_j\varepsilon_jx_j\|_{L^p(G,X)}\leq
2\|\sum_{j=1}^Nb_j\varepsilon_jx_j\|_{L^p(G,X)}.
\end{equation}
\end{lemma}
We now turn our attention to pseudo-sectorial and sectorial operators. The first mentioned class contains operators which are not necessarily injective. This will be important later on when periodic boundary conditions come into play temporarily.
\begin{definition}\label{defsec}
A closed densely defined linear operator $A$ on a Banach space $X$ is called 
{\em pseudo-sectorial} if there exist $C > 0$ and $\phi\in(0,\pi)$ such that
\[
	\Sigma_{\pi - \phi} := \{z \in \mathbb{C}\backslash\{0\};\ |\arg(z)| < \pi - \phi\} \subset \rho(-A)
\]
and
\begin{equation}\label{def_sec_t}
	\sup\{\| \lambda(\lambda+A)^{-1}\|;\ \lambda \in \Sigma_{\pi - \phi}\} \leq C.
\end{equation}
The number
\[
	\phi_A := \inf \big\{\phi;\ \rho(-A) \supset \Sigma_{\pi - \phi},\ 
	\sup\{\| \lambda(\lambda+A)^{-1}\|;\ \lambda \in \Sigma_{\pi - \phi}\} < \infty \big\}
\]
is called the {\em spectral angle} of $A$.
The class of pseudo-sectorial operators is denoted by $\Psi S(X)$.
A pseudo-sectorial operator $A$ is called {\em sectorial} if additionally $R(A)$ is dense in $X$ and $N(A) = \{0\}$.
The class of sectorial operators is denoted by $S(X)$.
\end{definition}
The following facts can be found in \cite[Chapter 2]{Haase-2006}. For each closed operator $A$ subject to condition \eqref{def_sec_t} one has $N(A) \cap \overline{R(A)} = \{0\}$, hence, density of $R(A)$ actually implies $N(A) = \{0\}$. If $X$ is reflexive we have $\overline{D(A)} = X$ and $X = N(A) \oplus \overline{R(A)}$. Thus, in particular $X$ being a UMD space \eqref{def_sec_t} implies $A \in \Psi S(X)$ and density of $R(A)$ and $N(A) = \{0\}$ are equivalent.

For $\sigma \in (0,\pi]$ let
\[
	\mathcal{H}^{\infty}(\Sigma_{\sigma}) :=
	\{f\colon \Sigma_{\sigma} \rightarrow \mathbb C;\ 
	 f \text{ is holomorphic, }	|f|_{\infty}^{\sigma} < \infty \}
\]
denote the commutative algebra of bounded, holomorphic functions
on $\Sigma_{\sigma}$. Here
$|f|_{\infty}^{\sigma} := \sup\{|f(z)|;\ z\in\Sigma_\sigma \}$.
Let $\rho(z) := \frac{z}{(1 + z)^2}$ and define the subalgebra
\begin{align*}
	\mathcal{H}^{\infty}_0(\Sigma_{\sigma}) :=
	\{ & f \in \mathcal{H}^{\infty}(\Sigma_{\sigma});\ 
	\exists\, C, \varepsilon > 0 \text{ such that } 
	|f(z)| \leq C|\rho(z)|^{\varepsilon} \text{ for all } z \in \Sigma_{\sigma} \}.
\end{align*}
Let $A$ be a pseudo-sectorial operator in $X$ with spectral angle $\phi_A$.
Pick $\sigma \in (\phi_A,\pi]$ and $\psi \in (\phi_A, \sigma)$.
The path $\Gamma := (\infty,0]e^{i\psi} \cup [0,\infty)e^{-i\psi}$
oriented counterclockwise, i.e.\ the positive real axis $\mathbb{R}_+$
lies to the left, stays with the only possible exception at zero in the
resolvent set of $A$. Hence, by Cauchy's integral formula and
the pseudo-sectoriality of $A$, the Bochner integral
\begin{equation}\label{Dunford}
	f(A) := \frac{1}{2\pi i} \int_{\Gamma}f(\mu) (\mu - A)^{-1}d\mu
\end{equation}
defines a well-defined element in $\mathcal L(X)$ for every  
$f \in \mathcal{H}_0^{\infty}(\Sigma_{\sigma})$. 
In view of pseudo-sectorial operators which do not have to be injective, the following definition of a bounded functional calculus is restricted to the class $\mathcal H_0^{\infty}(\Sigma_{\sigma})$.
\begin{definition}\label{Def_Hinf}
An operator $A \in \Psi S(X)$ is said to admit a \emph{bounded
$\mathcal H_0^{\infty}$-calculus} on $X$ if there exists a $\sigma > \phi_A$
such that
\begin{equation}\label{HUnendlich}
	\sup\{\|f(A)\|;\ f \in \mathcal H_0^{\infty}(\Sigma_{\sigma}),\ |f|_{\infty}^{\sigma} \leq 1\} \leq C_{\sigma}.
\end{equation}
We denote the class of pseudo-sectorial operators admitting a bounded
$\mathcal H_0^{\infty}$-calculus on $X$ by $\Psi\mathcal H^{\infty}(X)$.
The bound $C_{\sigma}$ in general depends on $\sigma$.
The infimum over all $\sigma>\phi_A$ such that this bound is finite
is called the \emph{$\mathcal{H}^{\infty}$-angle} of $A$ and is denoted by
$\phi_A^{\infty}$. If additionally $A \in S(X)$, then $A$ is said to admit a \emph{bounded
$\mathcal H^{\infty}$-calculus} on $X$ of $\mathcal{H}^{\infty}$-angle $\phi_A^{\infty}$. The class of such operators is denoted by $\mathcal H^{\infty}(X)$.
\end{definition}
\begin{remark}\label{Ext_Hinf}
For arbitrary $f \in \mathcal{H}^{\infty}(\Sigma_{\sigma})$ with $\rho$ from above and $A \in S(X)$ we set %
\[
	f(A) := \rho(A)^{-1}(\rho f)(A).
\]
This definition gives rise to a closed densely defined operator 
in $X$. Moreover, by Cauchy's theorem it is consistent with
the former one in the case $f \in \mathcal{H}_0^{\infty}(\Sigma_{\sigma})$ (see \cite[Chapter~2]{Haase-2006}). Due to the convergence lemma
(\cite[Lemma~2.1]{Cowling-Doust-McIntosh-Yagi-1996}, see also
\cite[Proposition~5.1.4]{Haase-2006}) the estimate \eqref{HUnendlich} extends to all $f \in \mathcal H^{\infty}(\Sigma_{\sigma})$ such that $|f|_{\infty}^{\sigma} \leq 1$.
\end{remark}
Recall our strategy to approach the non-linear thermoelastic plate equation via maximal regularity of an appropriate linearization $A$. To establish maximal regularity we have to prove $\mathcal R$-boundedness of the family $\{ \lambda(\lambda+A)^{-1};\ \lambda \in \Sigma_{\pi - \phi}\}$ of resolvents of $A$, where $\phi < \frac{\pi}{2}$ (\cite[Theorem 4.2]{Weis-2001}). If $A$ is closed and injective, this property is known as $\mathcal R$-sectoriality. Accordingly, $\mathcal R$-boundedness of the family $\{f(A);\ f \in \mathcal H_0^{\infty}(\Sigma_{\sigma}),\ |f|_{\infty}^{\sigma} \leq 1\}$ strengthens the property of a bounded $\mathcal H^\infty$-calculus.
\begin{definition}\label{Def_Rversions}
An operator $A \in \Psi S(X)$ is \emph{pseudo-$\mathcal R$-sectorial} if \eqref{def_sec_t} is valid with uniform norm boundedness replaced by $\mathcal R$-boundedness. Accordingly, it admits an \emph{$\mathcal R$-bounded
$\mathcal H_0^{\infty}$-calculus} on $X$ if \eqref{HUnendlich} is valid with uniform norm boundedness replaced by $\mathcal R$-boundedness.
The classes of these operators are denoted by $\Psi\mathcal RS(X)$ and $\Psi\mathcal {RH}^{\infty}(X)$.
If additionally $A \in S(X)$ holds, we speak of \emph{$\mathcal R$-sectoriality} and an \emph{$\mathcal R$-bounded
$\mathcal H^{\infty}$-calculus} on $X$ and denote these classes of operators by $\mathcal RS(X)$ and $\mathcal {RH}^{\infty}(X)$.
As in Definition \ref{Def_Hinf} we define the related angles $\phi_A^{\mathcal R}$ of $A \in (\Psi)\mathcal RS(X)$ and $\phi_A^{\mathcal R\infty}$ of $A \in (\Psi)\mathcal {RH}^{\infty}(X)$.
\end{definition}
\section{A combined Fourier multiplier theorem}\label{Sec_FM}

Throughout the section let $n \in \mathbb N$, $ 1 < p < \infty$, and let $E$ be a Banach space.
For $m \in L^\infty(\mathbb R^n ,\mathcal L(E))$ 
\[
	T_m f := \mathcal F^{-1} m \mathcal F f \quad (f \in \mathcal S(\mathbb R^n,E))
\]
is a well-defined mapping from $\mathcal S(\mathbb R^n,E)$ to $\mathcal S'(\mathbb R^n,E)$.
The function $m$ is called a {\em continuous, operator-valued, ($L^p$-)Fourier multiplier},
if $T_m f \in L^p(\mathbb R^n,E)$ for all $f \in \mathcal S(\mathbb R^n,E)$ and if there exists $C > 0$ such that
\[
	\| T_m f\|_{L^p(\mathbb R^n,E)} \leq C \|f\|_{L^p(\mathbb R^n,E)} \quad (f \in \mathcal S(\mathbb R^n,E)).
\]
In that case $T_m \in \mathcal L(L^p(\mathbb R^n,E))$ by density of $\mathcal S(\mathbb R^n,E) \subset L^p(\mathbb R^n,E)$ and $T_m$ is called the {\em Fourier multiplier operator associated with $m$}.
Given $M\colon \mathbb Z^n \rightarrow \mathcal L(E)$ the relation
\[
	( T_M f)\hat~  (k) = M(k) \hat f (k)
	\quad ( k \in \mathbb Z^n)
\]
for Fourier coefficients $\hat f(k)$ of $f$ defines a linear operator $T_M$ on the space of $E$-valued, trigonometric polynomials $\mathcal T(\mathcal Q_n,E)$. If there exists $C > 0$ such that
\[
	\| T_M f\|_{L^p(\mathcal Q_n,E)} \leq C \|f\|_{L^p(\mathcal Q_n,E)} \quad (f \in \mathcal T(\mathcal Q_n,E)),
\]
then $M$ is called a {\em discrete, operator-valued, ($L^p$-)Fourier multiplier}. In that case $T_M$ extends to $T_M \in \mathcal L(L^p(\mathcal Q_n,E))$ by density and $T_M$ is called the {\em Fourier multiplier operator associated with $M$}.

For the following important multiplier theorem in addition to partial derivatives of $m$ we will need
partial discrete derivatives of $M$ defined as
$ \Delta^{e_j} M(k) := 		M(k) - M(k - e_j)$. Here $e_j$ denotes the $j$-th unit vector in $\mathbb R^n$.
For arbitrary $\gamma \in \{0,1\}^n$ we set
\begin{equation}\label{disc_deriv_op}
 \Delta^{0} M = M, \quad 
 \Delta^{\gamma} M:=
 \Delta^{\gamma_1 e_1} \ldots\ \Delta^{\gamma_n e_n} M.
\end{equation}
Instead of $\gamma \in \{0,1\}^n$ we henceforth also write $0 \leq \gamma \leq 1$ or merely $\gamma \leq 1$. Furthermore, we indicate the variable (discrete) partial derivatives refer to, e.g.~we write $\partial_\xi^\gamma m(\xi)$ or $\Delta_k^\gamma M(k)$.
\begin{theorem}\label{Michlin}
Let $n_1, n_2, d \in \mathbb N$, $ 1 < p < \infty$, and let $\Lambda$ be an arbitrary index set.
Let
\[
	m_\lambda\colon \mathbb R^{n_1} \setminus \{0\} \times \mathbb Z^{n_2} \to \mathbb C^{d\times d};\ (\xi,k) \mapsto m_\lambda(\xi,k)
\]
and assume that $m_\lambda(\cdot,k) \in C^{n_1}(\mathbb R^{n_1} \setminus \{0\}, \mathbb C^{d\times d})$ for $k \in \mathbb Z^{n_2}$ and for $\lambda \in \Lambda$.
Assume there is $C > 0$ such that for each $\gamma = (\gamma_1,\gamma_2) \in \mathbb N_0^{n_1 + n_2}$ with $0 \leq \gamma \leq 1$
\begin{align*}
	\sup\big\{ |\xi^{\gamma_1} &
	\partial_\xi^{\gamma_1} m_\lambda(\xi,k)|;\ \lambda \in \Lambda,
	 (\xi,k) \in \mathbb R^{n_1} \setminus \{0\} \times \mathbb Z^{n_2} \big\} \leq C
\end{align*}
and, if $\gamma_2 \neq 0$,
\begin{align*}
	\sup\big\{ |\xi^{\gamma_1}k^{\gamma_2} &
	\partial_\xi^{\gamma_1} \Delta_k^{\gamma_2} m_\lambda(\xi,k)|;\ \lambda \in \Lambda,
	 (\xi,k) \in \mathbb R^{n_1} \setminus \{0\} \times \mathbb Z^{n_2} \setminus [-1,1]^{n_2} \big\} \leq C.
\end{align*}	
Then $\tilde m_{(\lambda,k)}(\xi) := m_\lambda(\xi,k)$ for $\xi \in \mathbb R^{n_1} \setminus \{0\}$ defines a continuous Fourier multiplier for all $(\lambda,k) \in \Lambda \times \mathbb Z^{n_2}$. Moreover, $M_{m_\lambda}(k) := \mathcal F^{-1}{\tilde m_{(\lambda,k)}} \mathcal F$ for $k \in \mathbb Z^{n_2}$ defines a discrete operator-valued Fourier multiplier and 
\[
	\{T_{M_{m_\lambda}};\ \lambda \in \Lambda\} \subset \mathcal L
	\big( L^p(\mathbb R^{n_1} \times \mathcal Q_{n_2})\big)
\]
is $\mathcal R$-bounded with $\mathcal R$-bound depending on $C$, $p$, $n_1$, and $n_2$ only.
\end{theorem}
\begin{proof}
The first assertion follows from the classical Michlin theorem. %
Due to \cite[Theorem 3.2]{Girardi-Weis-2003} $\mathcal R$-boundedness of
$\big\{\mathcal F^{-1}{\tilde m_{(\lambda,k)}} \mathcal F;\ \lambda \in \Lambda,\ 
	k \in \mathbb Z^{n_2}\big\}$ and
\begin{align*}
	\big\{ k^{\gamma_2} &
	\Delta_k^{\gamma_2} \mathcal F^{-1}{\tilde m_{(\lambda,k)}} \mathcal F;\ \lambda \in \Lambda,\ 
	k \in \mathbb Z^{n_2} \setminus [-1,1]^{n_2}, \gamma_2 \neq 0\big\} 
\end{align*}	
follows. This in turn provides the condition on $M_{m_\lambda}$ to define a discrete operator-valued Fourier multiplier (see \cite{Nau-Diss-2012} or  \cite{Arendt-Bu-2002}, \cite{Bu-Kim-2004}, and \cite{Bu-2006} or \cite{Strkalj-Weis-2007}). As in \cite{Girardi-Weis-2003} (see also \cite{Bu-Kim-2004} and \cite{Bu-2006}) we deduce $\mathcal R$-boundedness of
$\{T_{M_{m_\lambda}};\ \lambda \in \Lambda\} \subset \mathcal L
	\big( L^p(\mathcal Q_{n_2}, L^p(\mathbb R^{n_1}))\big)$.
\end{proof}

\section{Analysis of the linearized Problem - Proof of Theorem \ref{Hinf}}\label{Sec_Proofs}

Recall the matrix $M$ from \eqref{Def_M} and let $n = n_1 + n_2$. In what follows we investigate the symbol
\begin{equation}\label{eq:symbol}
	a \colon \mathbb R^{n}\to \mathbb R^{3\times 3};\ a(\zeta) := M |\zeta|^2.
\end{equation}
As in \cite{Naito-Shibata-2009} we infer parameter-ellipticity in the sense of \cite{Denk-2003} of angle less than $\frac{\pi}{2}$. More precisely, there exists $\phi < \frac{\pi}{2}$ such that all eigenvalues of $a(\zeta)$ lie in the sector $\Sigma_\phi$, i.e.
\[
	\sigma(a(\zeta)) \subset \Sigma_{\phi} \qquad (\zeta \in \mathbb R^n\setminus \{0\}).
\]
First we consider \eqref{TEP_lin_first_order} with $\Omega$ replaced by $\tilde \Omega := \mathbb R^{n_1} \times \mathcal Q_{n_2}$ and prove pseudo-$\mathcal R$-sectoriality of the $L^p(\tilde \Omega)^3$-realization $\tilde A$ of the matrix differential operator $A(D)$ defined by
\begin{equation}\label{Real_tilde}
\begin{aligned}
 D(\tilde A) & :=  \bigg( \bigcap_{\ell = 0}^2 W^{\ell,p}\big(\mathbb R^{n_1},W_{per}^{2-\ell,p}(\mathcal Q_{n_2}) \big)\bigg)^3,\\
 \tilde A u & :=  A(D) u \quad (u \in D(\tilde A)).
\end{aligned}
\end{equation}
To this end, with $\alpha \in \mathbb N_0^n$, $|\alpha| \leq 2$, let
\[
	\kappa(\lambda,\zeta) := \lambda^{1-\frac{|\alpha|}{2}}\zeta^\alpha\big(\lambda + a(\zeta)\big)^{-1} \qquad (\lambda \in \Sigma_{\pi - \phi},\ \zeta \in \mathbb R^n).
\]
Then $\kappa$ is quasi-homogeneous of degree $(2,1)$, thus there exists $C > 0$ such that $\kappa(\lambda, \zeta) \leq C$ for all $\lambda \in \Sigma_{\pi - \phi}$ and $\zeta \in \mathbb R^n$. Recall that $0 \notin \Sigma_{\pi - \phi}$. Hence, for 
$\zeta = (\zeta_1,\zeta_2)$ and the choice $\zeta = (\xi,k)$ we see that $m_\lambda(\xi,k) := \kappa(\lambda,(\xi,k))$ is uniformly bounded. The same assertion is valid for $\zeta^\gamma \partial_\zeta^\gamma \kappa(\lambda,\zeta)$ where $\gamma = (\gamma_1,\gamma_2) \in \mathbb N_0^{n_1 + n_2}$, $0 \leq \gamma \leq 1$, and $\zeta \in \mathbb R^n \setminus \{0\}$. %
Given $k \in \mathbb Z^{n_2} \setminus [-1,1]^{n_2}$ observe that
\[
	\Delta_k^{\gamma_2}\partial_\xi^{\gamma_1}m_\lambda(\xi,k)
	= \partial^{\gamma}_{\zeta}\kappa(\lambda,(\xi,\eta))
\]
for some $\eta \in [k_1, k_1 + \gamma_2^{(1)}] \times \ldots \times [k_{n_2}, k_{n_2} + \gamma_2^{(n_2)}]$ by the mean value theorem. Therefore we have
\[
	|\xi^{\gamma_1}k^{\gamma_2} \partial_\xi^{\gamma_1} \Delta_k^{\gamma_2}m_\lambda(\xi,k)|
	= |\xi^{\gamma_1}k^{\gamma_2}\partial^{\gamma}_{\zeta}\kappa(\lambda,(\xi,\eta))|
	\leq C |(\xi,\eta)^{\gamma}\partial^{\gamma}_{\zeta}\kappa(\lambda,(\xi,\eta))|
\]
uniformly in $\xi \in \mathbb R^{n_1} \setminus \{0\}$, $k \in \mathbb Z^{n_2} \setminus [-1,1]^{n_2}$, and $\lambda \in \Sigma_{\pi - \phi}$. 
Hence, $m_\lambda$ defines a combined Fourier multiplier in the sense of Theorem \ref{Michlin}. Thus we have proved
\begin{proposition}\label{PsiRS_tilde}
Let $1 < p< \infty$. Then $\tilde A \in \Psi\mathcal RS(L^p(\tilde \Omega)^3)$ and $\phi_{\tilde A}^{\mathcal R} < \frac{\pi}{2}$.
Moreover, for each $\phi > \phi_{\tilde A}^{\mathcal R}$ it holds that
\begin{align*}%
	&\mathcal R(\{\lambda^{1-\frac{|\alpha|}{2}}\partial^\alpha(\lambda + \tilde A)^{-1};\ \lambda \in \Sigma_{\pi - \phi},\ \alpha \in \mathbb N_0^{n_1 + n_2},\  |\alpha| \leq 2\}) < \infty.
\end{align*}
\end{proposition}

To show that $\tilde A$ admits an $\mathcal R$-bounded pseudo-$\mathcal H^\infty$-calculus of angle $\phi_{\tilde A}^{\mathcal R \infty} \leq \phi_{\tilde A}^\mathcal R$ let $\phi > \phi_{\tilde A}^{\mathcal R}$ be arbitrary and $f \in \mathcal H_0^\infty(\Sigma_\phi)$. With $\Gamma$ as in \eqref{Dunford}
let
\[
	h_f(\zeta) := \frac{1}{2\pi i} \int_{\Gamma}f(\mu) (\mu - a(\zeta))^{-1}d\mu.
\]
Thanks to the behavior of $f \in \mathcal H_0^\infty(\Sigma_\phi)$ close to zero this is a well-defined element in $\mathbb C^{3\times 3}$ for all $\zeta \in \mathbb R^n$.
In \cite[Theorem 5.5]{Denk-2003} existence of $C > 0$ is shown such that 
\[
	|\zeta|^{|\gamma|} |\partial^\gamma_\zeta h_f(\zeta)| \leq C |f|^\phi_\infty \qquad (\zeta \in \mathbb R^{n} \setminus \{0\})
\]
for $\gamma = (\gamma_1,\gamma_2) \in \mathbb N_0^{n_1 + n_2}$, $0 \leq \gamma \leq 1$. As above for $\gamma \neq 0$ we deduce 
\[
	|\xi^{\gamma_1}k^{\gamma_2} \partial_\xi^{\gamma_1} \Delta_k^{\gamma_2} h_f((\xi,k))| 
	\leq C |(\xi,\eta)|^{|\gamma|}|\partial^{\gamma}_{\zeta}h_f((\xi,\eta))|
	\leq C |f|^\phi_\infty
\]
uniformly in $\xi \in \mathbb R^{n} \setminus \{0\}$, $k \in \mathbb Z^{n_2} \setminus [-1,1]^{n_2}$.
Thus for all $f \in \mathcal H_0^\infty(\Sigma_\phi)$ such that $|f|^\phi_\infty \le 1$ the symbols $h_f$ define combined Fourier multipliers in the sense of Theorem \ref{Michlin} subject to a uniform bound $C >0$. Thus we have proved
\begin{proposition}\label{PsiHinf_tilde}
Let $1 < p< \infty$. Then we have $\tilde A \in \Psi\mathcal {RH}^\infty(L^p(\tilde \Omega)^3)$ and $\phi_{\tilde A}^{\mathcal R\infty} < \frac{\pi}{2}$.
\end{proposition}

Now we are in the position to treat the $L^p(\Omega)^3$-realization $A$ as defined in \eqref{Real} by means of reflection.

After a change of coordinate directions let $\Omega = \mathbb R^{n_1} \times  (0,\infty)^{n_3} \times (0,\pi)^{n_2}$ and $u \in L^p(\Omega)^3$. We define the extension $\mathfrak E u := \mathfrak E_2 \mathfrak E_1 u$ to $\tilde \Omega := \mathbb R^{n_1 + n_3} \times \mathcal Q_{n_2}$ in two steps. First, let $\mathfrak E_1$ extend $u$ to $\mathbb R^{n_1+n_3} \times (0,\pi)^{n_2}$ such that $\mathfrak E_1 u$ is odd with respect to $x_j = 0$ in each coordinate direction $x_j$ for $j = n_1 + 1, \ldots, n_1 + n_3$. Afterwards let $\mathfrak E_2$ extend $\mathfrak E_1 u$ to $\tilde \Omega$ such that $\mathfrak E_2 \mathfrak E_1 u$ is odd with respect to $x_j = \pi$ for $ j = n_1 + n_3 + 1,\ldots, n_1 + n_2 + n_3$.
This construction gives rise to an extension operator $\mathfrak E \in \mathcal L(L^p(\Omega)^3,L^p(\tilde \Omega)^3)$.

\begin{proof}[Proof of Theorem \ref{Hinf}]
Let $\mathfrak R \in \mathcal L(L^p(\tilde \Omega)^3,L^p(\Omega)^3)$ denote the operator of restriction and let $\tilde A$ be defined as in \eqref{Real_tilde} with $n_1$ replaced by $n_1 + n_3$. Then $\rho(\tilde A) \subset \rho(A)$ and 
\begin{equation}\label{Ext_resolvent}
	\big(\lambda - A\big)^{-1} = \mathfrak R \big(\lambda -\tilde A\big)^{-1} \mathfrak E
	\quad (\lambda \in \rho(\tilde A))
\end{equation}
by construction. Here we have used that 
\begin{equation}\label{Ext_Delta}
	u \in D(A) \Rightarrow \mathfrak E u \in D(\tilde A)
	\text{ and } \tilde A (\mathfrak E u) = \mathfrak E (A u).
\end{equation}
In turn we have made use of the fact that $u \in W^{2,p}(\tilde \Omega)^3$ being odd with respect to $x_j = 0$ for $j = n_1 + 1,\ldots,n_1 + n_3$, implies Dirichlet conditions of $u|_{\Omega}$ on $\{x \in \overline \Omega;\ x_j = 0 \}$ for each $j = n_1 + 1,\ldots,n_1 + n_3$. Accordingly, $u \in W^{2,p}(\tilde \Omega)^3$ being odd with respect to $x_j = \pi$ and $2\pi$-periodic with respect to $x_j$ for $j = n_1 + n_3 + 1,\ldots,n_1 + n_2 + n_3$, yields Dirichlet conditions of $u|_{\Omega}$ on $\{x \in \overline \Omega;\ x_j = 0 \text{ or } x_j = \pi \}$ for each $j = n_1 + n_3 + 1,\ldots,n_1 + n_2 + n_3 $.

As Propositions \ref{PsiRS_tilde} and \ref{PsiHinf_tilde} apply to $\tilde A$ with $n_1$ being replaced by $n_1 + n_3$, relation \eqref{Ext_resolvent} and Lemma \ref{HintSumme} prove 
$A \in \Psi\mathcal {RH}^\infty(L^p(\Omega)^3)$ and $\phi_{A}^{\mathcal R\infty} < \frac{\pi}{2}$.
Moreover, for each $\phi > \phi_{A}^{\mathcal R \infty}$ it holds that
\begin{align}\label{R_bounds}
	&\mathcal R(\{\lambda^{1-\frac{|\alpha|}{2}}\partial^\alpha(\lambda + A)^{-1};\ \lambda \in \Sigma_{\pi - \phi},\ \alpha \in \mathbb N_0^{n_1 + n_2 + n_3},\  |\alpha| \leq 2\}) < \infty.
\end{align}

It remains to prove injectivity of $A$ and that $n_2 \neq 0$ implies $0 \in \rho(A)$.

First assume $n_2 \neq 0$. Then $\mathfrak E f$ is odd with respect to $x_{j} = \pi$ for each $j = n_1 + n_3 + 1,\ldots,n_1 + n_2 + n_3$ and each $f \in L^p(\Omega)^3$. Hence, 
\[
	\mathfrak E f \in L^p_{(0)}(\mathcal Q_{n_2}, L^p(\mathbb R^{n_1 + n_3}))^3:= \{ u \in L^p(\mathcal Q_{n_2}, L^p(\mathbb R^{n_1 + n_3}))^3;\ \int_{\mathcal Q_{n_2}} u = 0\}.
\]
Let us consider the equation $\tilde A u = f$ in $L^p_{(0)}(\mathcal Q_{n_2}, L^p(\mathbb R^{n_1 + n_3}))^3$. We calculate the Fourier coefficients with respect to $\mathcal Q_{n_2}$ and apply the operator-valued Fourier multiplier result from \cite{Bu-2006} or \cite{Strkalj-Weis-2007}. This proves unique solvability of $\tilde A u = f$, i.e., invertibility of $\tilde A$ restricted to $L^p_{(0)}(\mathcal Q_{n_2}, L^p(\mathbb R^{n_1 + n_3}))^3$. Now \eqref{Ext_resolvent} shows $0 \in \rho(A)$.

Now let $n_2 = n_3 = 0$, $n_1 \neq 0$  and assume $A u = 0$. For $\lambda \in (0,\infty)$ this gives $(\lambda + A)u = \lambda u$, hence for some constant $C \ge 0$ one has $\|\partial^\beta u\|_{L^p(\mathbb R^{n_1})^3} \leq C\lambda\|u\|_{L^p(\mathbb R^{n_1})^3}$ for each $|\beta| = 2$ by the $\mathcal R$-boundedness result \eqref{R_bounds}. As $\lambda \in (0,\infty)$ is arbitrary for each $|\beta| = 2$ this proves $\|\partial^\beta u\|_{L^p(\mathbb R^{n_1})^3}= 0$ and so $u = 0$. Thus $A$ is injective. Injectivity in case of arbitrary $n_1, n_3$ now follows from \eqref{Ext_Delta}. Altogether we have proved $A \in \mathcal {RH}^\infty(L^p(\Omega)^3)$ with $\phi_{A}^{\mathcal R\infty} < \frac{\pi}{2}$ and that $n_2 \neq 0$ implies $0 \in \rho(A)$.
\end{proof}
\section{Existence, Uniqueness and Regularity of the Solution}

In this section we apply our results on maximal regularity of the linearized problem to show local existence, uniqueness and regularity of the strong solution. We will need the fractional Sobolev spaces $W^s_p(\Omega)$ defined in Section~\ref{Sec_FSandCO}. We introduce the maximal regularity space $MR_p(0,T) \coloneqq W^{1,p}((0,T), L^p(\Omega)^3) \cap L^p((0,T), D(A))$ to shorten our notation. In the following we will make frequent use of the embedding $MR_p(0,T) \hookrightarrow C([0,T],I_p(A))$~\cite[Lemma~4.1]{Arendt-Duelli-2006}, in particular the evaluation at time zero is well-defined.

\begin{proof}[Proof of Theorem~\ref{Non-lin}]
	Let $T > 0$. We define the linear operator $L_T$ obtained from the linearized equation by
		\begin{align*}
			L_T: MR_p(0,T) & \to L^p((0,T);L^p(\Omega)^3) \times I_p(A) \\
			U & \mapsto \begin{pmatrix} U_t + AU, & U(0) \end{pmatrix}^{T}.
		\end{align*}
	By the maximal regularity of $A$ (Corollary~\ref{Cor_MR}), $L_T$ is an isomorphism. Further, given an initial value $U_0 \in I_p(A)$ we define $G_T(U) = \begin{pmatrix} \Phi(U), & U_0 \end{pmatrix}$ with the same domain and codomain as $L_T$. Observe that $U$ being a strong solution for the initial value $U_0$ in the time interval $[0, T)$ is equivalent to $U$ being a fixed point of $L_T^{-1}G_T$, i.e. $U = (L_T^{-1}G_T)(U)$. We want to apply the Banach fixed point theorem to show the existence of such a fixed point. 
	
	Observe that
		\[ \Delta (U_1^3) = 3 \Delta U_1 U_1^2 + 6 \nabla U_1 \nabla U_1 U_1. \]
	We now show that $U_1, \nabla U_1 \in C_b((0,T); C_b(\Omega))$ which implies by using the Hölder inequality that $G_T$ is well-defined. Indeed, one has for $s \in (0,1)$ by~\cite[Section~7.32]{Adams-2003}
		\begin{align*} 
			\MoveEqLeft W^{1,p}((0,T), L^p(\Omega)) \cap L^p((0,T), W^{2,p}(\Omega)) \\
			& \hookrightarrow [L^p((0,T), W^{2,p}(\Omega)), W^{1,p}((0,T), L^p(\Omega))]_{s,p} =  W^{s}_p((0,T), W^{2(1-s)}_p(\Omega)).
		\end{align*}
	Now for $s$ such that $sp > 1$ and $(1-2s)p > n$ the Sobolev embedding theorem~\cite[Theorem~7.34]{Adams-2003} yields
		\[ \nabla U_1 \in W^{s}_p((0,T), W^{1 - 2s}_p(\Omega)) \hookrightarrow C_b((0,T), C_b(\Omega)). \]
	We let $s = \frac{1+\epsilon}{p}$ for $\epsilon > 0$. Then $sp > 1$ and $(1-2s)p > 1$ for $\epsilon$ small enough as we have assumed $p > n + 2$ and the embedding holds.
	
	We have already seen that $U$ being a strong solution in $[0,T)$ is equivalent to $U$ being a solution of $L_TU = G_TU$. Choose $T_0 > 0$ and let $U^* = L_{T_0}^{-1}(0, U_0) \in MR_p(0,T_0)$. Under the substitution $V = U - U^*$ we obtain for $T \le T_0$
		\begin{equation*}
			L_T(V) = G_T(V + U^*) - L_T(U^*) = \begin{pmatrix} \Phi(V + U^*) \\ U_0 \end{pmatrix} - \begin{pmatrix} 0 \\ U_0 \end{pmatrix} = \begin{pmatrix} \Phi(V + U^*) \\ 0 \end{pmatrix}.
		\end{equation*}
	This way we have reduced the problem to the fixed point problem $V = L_{T}^{-1} \begin{pmatrix} \Phi(V + U^*), 0 \end{pmatrix}^T$ in the closed subspace $MR_{p,0}(0,T) \coloneqq \{ U \in MR_p(0,T): U(0) = 0 \}$ of functions of trace zero. The main advantage is that the norm of the inverse of $L_T$ restricted to $MR_{p,0}(0,T)$ is uniformly bounded in $[0,T_0]$, say by $M$. In the following we will not distinguish between $L_T$ and its restriction in our notation.
	Since $\Phi: MR_p(0,T) \to L^p((0,T), L^p(\Omega)^3)$ is continuously differentiable and satisfies $\Phi'(0) = 0$, there is an $r > 0$ such that $\norm{\Phi^{\prime}(W)} \le \frac{1}{2M}$ for all $W \in B_{2r}(0)$. Now it follows from the dominated convergence theorem that for $T \in [0,T_0]$ sufficiently small one has $\norm{U^*|_{[0,T]}} < r$ in $MR_p(0,T)$. Then for $V \in B_r(0)$ we have
		\[ \norm{L_T^{-1} \Phi(V + U^*)} \le \norm{L_T^{-1}} \sup_{W \in B_{2r}(0)} \norm{\Phi^{\prime}(W)} \norm{V + U^*} \le \frac{2r \norm{L_T^{-1}}}{2M} \le r. \]
	Thus $L_T^{-1} \Phi(\cdot + U^*) B_r(0) \subset B_r(0)$. Moreover, for $V_1, V_2 \in B_r(0)$ we have
		\begin{align*} \norm{L_T^{-1} \Phi(V_1 + U^*) - L_T^{-1} \Phi(V_2 + U^*)} & \le \norm{L_T^{-1}} \sup_{W \in B_{2r}(0)} \norm{\Phi'(W)} \norm{V_1 - V_2} \\
		& \le \frac{1}{2} \norm{V_1 - V_2}. \end{align*}
	Now, the Banach fixed point theorem gives us a unique fixed point $V$. Then $V + U^*$ is the unique local strong solution of the system in $[0,T)$. We have shown that for an arbitrary initial value there exists a local solution for sufficiently small times. Given an arbitrary time $T > 0$, one can show by employing the Banach fixed point theorem again and using similar estimates for the non-restricted mapping $L_T^{-1}G_T$ that a unique strong solution exists for $[0,T)$ provided the initial values are small enough.
	
	We conclude with the real-analytic dependence of the solution $U \in MR_p(0,T)$ on the initial value $U_0$. This is proven via the implicit function theorem. Again, we only demonstrate the case for arbitrary initial values and small times. The case of a fixed time and small initial values is proven analogously. We define
		\begin{align*}
			\Psi: MR_p(0,T) \times I_p(A) & \to MR_p(0,T) \\
			(V, U_0) & \mapsto L_T^{-1} \begin{pmatrix} \Phi(V + L_T^{-1}(0, U_0)), 0 \end{pmatrix}^T - V.
		\end{align*}
	Then $U = V + L_T^{-1}(0, U_0)$ is the solution for the initial value $U_0$ if and only if $\Psi(V, U_0) = 0$. Further notice that $\Psi$ is real-analytic. Given a solution $U$ to the initial value $U_0$ one has besides $\Psi(V, U_0) = 0$
		\begin{equation} D_1\Psi(V,U_0)(W,0) = L_T^{-1} ((\Phi^{\prime}(V + L_T^{-1}(0, U_0))(W), 0)^T) - W. \label{eq:inv_derv} \end{equation}
	Now, as above we see that $\norm{\Phi^{\prime}(V + L_T^{-1}(0, U_0))} \le (2M)^{-1}$. An application of the von Neumann series shows that~\eqref{eq:inv_derv} is invertible. By the implicit function theorem there exists a real-analytic function $g: I_p(A) \to MR_p(0,T)$ such that $\Psi(g(\tilde{U}_0), \tilde{U}_0) = 0$ for all $\tilde{U}_0$ in a small neighbourhood of $U_0$. Then $g(\tilde{U}_0)$ is the unique solution for the initial value $\tilde{U}_0$ showing the real-analytic dependence on the initial value.
\end{proof}

\begin{proof}[Proof of Theorem~\ref{Analytic}]
	Let $U$ be the unique strong solution in $[0, T_0)$ for the initial value $U_0$. We use a tweaked version of a trick going back to Masuda and Angenent from~\cite{Escher-Pruess-Simonett-2003-2} to prove real-analyticity of $U$. Choose $(t_0, x_0) \in (0,T_0) \times \Omega$ and $\epsilon_0 > 0$ such that $[t_0 - 3\epsilon_0, t_0 + 3\epsilon_0] \subset (0,T_0)$ and $\overline{B}_{3\epsilon_0}(x_0) \subset \Omega$. Further, we choose smooth real cut-off functions $\zeta \in C_0^{\infty}((t_0 - 2\epsilon_0, t_0 + 2\epsilon_0))$ with $\zeta \equiv 1$ on $[t_0 - \epsilon_0, t_0 + \epsilon_0]$ and $0 \le \zeta \le 1$ and $\xi \in C_0^{\infty}(B_{2\epsilon}(x_0))$ with $\xi \equiv 1$ on $\overline{B}_{\epsilon_0}$ and $0 \le \xi \le 1$.
	We now define parameterized coordinates by
		\[ \Theta_{\lambda, \mu}(t,x) \coloneqq (t + \lambda \zeta(t), x + \mu\zeta(t)\xi(x)).  \]
	Then for sufficiently small $(\lambda, \mu)$, say $(\lambda, \mu) \in B_{r_0}(0) \subset \R \times \R^n$, $\Theta_{\lambda, \mu}$ is a diffeomorphism in $(0,T) \times \Omega$~\cite[p.~15]{Escher-Pruess-Simonett-2003-2}. We denote by $\Theta^{*}_{\lambda, \mu}$ the pull-back map on functions. Note that for such $(\lambda, \mu)$ the pull-backs $U_{\lambda, \mu} \coloneqq \Theta^{*}_{\lambda, \mu} U = T_{\mu}\theta^*_{\lambda} U$, where $\theta_{\lambda}^{*}$ and $T_{\mu}$ are the pullbacks on the time and space coordinates respectively, satisfy the boundary conditions
	\[
		U_{\lambda, \mu}(0, \cdot) = U(0, \cdot) = U_0 \quad \text{and} \quad U_{\lambda, \mu}(t,x) = U(t + \lambda\zeta(t), x) = 0 \text{ on } (0,T_0) \times \partial \Omega.
	\]
	More precisely, one has $U_{\lambda, \mu} - U = 0$ outside a compact subset of $(0,T_0) \times \Omega$. It is shown in~\cite[Proposition~5.3]{Escher-Pruess-Simonett-2003-2} that
		\[ U_{\lambda, \mu} \in W^{1,p}((0,T_0), L^p(\Omega)^3) \cap L^p((0,T_0), (W^{2,p}(\Omega) \cap W_0^{1,p}(\Omega))^3) \]
	and that $U_{\lambda, \mu}$ satisfies the equation
	 \begin{align*}
	 	\partial_t U_{\lambda, \mu} & = (1 + \zeta'\lambda) T_{\mu} \theta_{\lambda}^{*} \partial_t U + B_{\mu} U_{\lambda, \mu} \\
		& = (1 + \zeta' \lambda) T_{\mu} \theta_{\lambda}^* (-AU + \Phi(U)) + B_{\mu} U_{\lambda, \mu} \\
		& = -(1 + \zeta' \lambda) T_{\mu} A (\theta^{*}_{\lambda}U) + (1 + \zeta' \lambda) T_{\mu} \Phi(\theta_{\lambda}^* U) + B_{\mu} U_{\lambda, \mu} \\
		& = -(1 + \zeta' \lambda) A U_{\lambda, \mu} + (1 + \zeta'\lambda) \Phi(U_{\lambda, \mu}) + B_{\mu}U_{\lambda, \mu} + (1 + \zeta'\lambda) C_{\mu} U_{\lambda, \mu},
	 \end{align*}
	 where $B_{\mu}$ and $C_{\mu}$ are real-analytic functions in $[0, T_0)$ with values in $\mathcal{L}((W^{2,p}(\Omega) \cap W_0^p(\Omega))^3, L^p(\Omega)^3)$ and $B_0 = C_0 = 0$. Some comments on the manipulations above are in order. Notice that in the second last equation we have used the fact that $\theta_{\lambda}^*$ commutes with $\Delta$ and therefore also with $A$ and $\Phi$. In the last equation we need to calculate the terms occuring when one interchanges $T_{\mu}$ with $\Phi$ and $A$. We do this first for $\Delta$ and a scalar-valued $u \in W^{2,p}(\Omega)$. Here we obtain in the weak sense
	 	\begin{align*}
			\Delta(T_{\mu} u) & = T_{\mu} \Delta u + \mu \biggl( \cdots \biggr),
		\end{align*}
	where the second summand is a real-analytic mapping with values in $\mathcal{L}((W^{2,p}(\Omega) \cap W_0^p(\Omega))^3, L^p(\Omega)^3)$. Hence, we obtain for $A = -M \Delta$ and $\Phi$
		\[ T_{\mu} A (\theta^*_{\lambda} U) + T_{\mu} \Phi(\theta_{\lambda}^*U) = AU_{\lambda, \mu} + \Phi(U_{\lambda,\mu}) + C_{\mu} U_{\lambda, \mu} \]
	with $C_{\mu}$ as above. Since $\Theta_{\lambda,\mu}$ is a diffeomorphism and by the above calculations, $U$ is a strong solution of the system if and only if $U_{\lambda, \mu}$ is a strong solution of
	\begin{equation}
		\label{eq:lambda}
		\begin{array}{r@{\quad=\quad}l@{\quad}l}
			W_t + (1 + \zeta^{\prime} \lambda) AW & (1+\zeta^{\prime} \lambda) \Phi(W) + D_{\mu} W & \text{in } (t,x) \in (0, T) \times \Omega, \\
			W(0,x) & U_0(0,x) & \text{in } x \in \Omega, \\
			W(t,x) & 0 & \text{on } (0,T) \times \partial \Omega,
		\end{array}
	\end{equation}
	where we have set $D_{\mu} \coloneqq B_{\mu} + (1+\zeta^{\prime}\lambda) C_{\mu}$. Motivated by the calculations we set
		\begin{align*}
			\Psi: B_{r_0}(0) \times MR_p(0,T_0) & \to L^p((0,T_0), L^p(\Omega)^3) \times I_p(A) \\
			(\lambda, \mu, W) & \mapsto \begin{pmatrix} W_t + (1+ \lambda \zeta^{\prime})(AW - \Phi(W) - C_{\mu}W) + B_{\mu} W \\ W(0) - U_0 \end{pmatrix}.
		\end{align*}
	Then $W$ is a solution of~\eqref{eq:lambda} if and only if $\Psi(\lambda, \mu, W) = 0$.  Further, note that $\Psi$ is a real-analytic mapping. We want to apply the implicit function theorem. For this notice that $\Psi(0,0,U) = 0$. The partial derivative with respect to the third component in $(0,0,U)$ is
		\begin{equation*}
			D_3\Psi(0,0,U)(0,0,V) = \begin{pmatrix} V_t + AV - \Phi'(U)(V) \\ V(0) \end{pmatrix} = L_{T_0} \biggl( \Id - L_{T_0}^{-1} \begin{pmatrix} \Phi'(U) \\ 0 \end{pmatrix} \biggr)(V).
		\end{equation*}
	We have seen in the proof of Theorem~\ref{Non-lin} that $\norm{\Phi'(U)} \le (2M)^{-1}$. Hence, by the von Neumann series, $D_3\Psi(0,0,U)(0,0,\cdot)$ is invertible. Now the implicit function theorem yields the existence of $g \in C^{\omega}(B_{\epsilon}(0), MR_p(0,T_0))$ for some $\epsilon_0 > \epsilon > 0$ with $\Psi(\lambda, \mu, g(\lambda, \mu)) = 0$ for all $(\lambda, \mu) \in B_{\epsilon}(0)$. As \eqref{eq:lambda} has a unique solution, we deduce $g(\lambda, \mu) = U_{\lambda, \mu}$. Because of the embedding $MR_p(0,T_0) \hookrightarrow C_b((0,T_0) \times \Omega)^3$ (cf.~proof of Theorem~\ref{Non-lin}) the point evaluation in $(t_0, x_0)$ is a well-defined linear and therefore a real-analytic mapping. Hence, 
	\[ (\lambda, \mu) \mapsto U(t_0 + \lambda \zeta(t_0), x_0 + \mu \zeta(t_0) \xi(x_0)) = U(t_0 + \lambda, x_0 + \mu)  \]
	is real-analytic.
\end{proof}

\bibliographystyle{amsalpha}
\bibliography{lit_platte}
\end{document}